\documentclass[production,12pt]{jsg}
\usepackage{latexsym,amsmath,amsfonts,enumerate,amscd,amssymb,hyperref
} 

\newtheorem{thm}{Theorem}
\newtheorem{cor}[thm]{Corollary}
\newtheorem{prop}[thm]{Proposition}
\newtheorem{lemma}[thm]{Lemma}
\newtheorem{claim}{Claim}

\newtheorem{rmk}[thm]{Remark}
\newtheorem{rmks}[thm]{Remarks}
\newtheorem*{Rmk}{Remark}
\newtheorem*{Rmks}{Remarks}

\newtheorem*{Pb}{Problem}

\newtheorem{xpl}[thm]{Example}
\newtheorem*{Xpl}{Example}

\newenvironment{pf}[1][Proof.]{\noindent \textbf{#1:} }{}
\newenvironment{enui}{\begin{enumerate}[(i)]}{\end{enumerate}}
\newenvironment{enua}{\begin{enumerate}[(a)]}{\end{enumerate}}

\newcommand\R{\mathbb{R}}
\newcommand\RP{\R{\operatorname{P}}}
\newcommand\C{\mathbb{C}}
\newcommand\CP{\C{\operatorname{P}}}
\newcommand\N{\mathbb{N}}
\newcommand{\Z}{\mathbb{Z}}
\newcommand\al{\alpha}
\newcommand\ga{\gamma}
\newcommand\lam{\lambda}
\newcommand\om{\omega}
\newcommand\sub{\subseteq}

\newcommand\x{\times}
\newcommand\HH{\mathcal{H}}
\newcommand\D{\mathbb{D}}
\newcommand\id{\operatorname{id}}

\newcommand\nn{\nonumber}
\newcommand\im{\operatorname{im}}
\newcommand\F{{\mathcal{F}}}
\newcommand\A{{\mathcal{A}}}

\newcommand{\BAR}[1]{{\overline{#1}}}
\newcommand\wt[1]{{\widetilde{#1}}}
\newcommand\wo{\setminus}
\newcommand\wrt{w.r.t.~}
\newcommand\LL{{\mathcal{L}}}
\newcommand\const{\equiv}
\renewcommand\phi{\varphi}
\newcommand\si{\sigma}
\newcommand\Diff{{\operatorname{Diff}}}
\newcommand\Wlog{w.l.o.g.~}
\newcommand\WLOG{W.l.o.g.~}
\newcommand\pr{\operatorname{pr}}
\newcommand\dd{\partial}
\newcommand\iso{\cong}
\newcommand\Lie{\operatorname{Lie}}
\newcommand\area{\operatorname{area}}
\newcommand\FS{{\operatorname{FS}}}
\newcommand\corank{\operatorname{corank}}
\newcommand\then{\Rightarrow}
\newcommand\follows{\Leftarrow}
\newcommand\cont{\supseteq}
\newcommand\pt{\operatorname{pt}}

\title[Strict Arnold chords and coisotropic submanifolds]{On the strict Arnold chord property and coisotropic submanifolds of complex projective space} 

\author[Fabian Ziltener]{Fabian Ziltener\\
Utrecht University, Mathematics Department, Budapestlaan 6, 3584CD Utrecht, The Netherlands\\
e-mail: f.ziltener@uu.nl
} 

\begin{document}

\begin{abstract} Let $\al$ be a contact form on a manifold $M$, and $L\sub M$ a closed Legendrian submanifold. I prove that $L$ intersects some characteristic for $\al$ at least twice if all characteristics are closed and of the same period, and $\al$ embeds nicely into the product of $\R^{2n}$ and an exact symplectic manifold. As an application of the method of proof, the minimal action of a regular closed coisotropic submanifold of complex projective space is at most $\pi/2$. This yields an obstruction to presymplectic embeddings, and in particular to Lagrangian embeddings.
\end{abstract}

\maketitle
\tableofcontents
\section{Main results}

\subsection{The strict chord property}
Let $M$ be a manifold (possibly noncompact or with boundary) and $\al$ a contact form on $M$. We say that $(M,\al)$ has the \emph{strict chord property} iff for every nonempty closed
\footnote{This means ``compact and without boundary''.}
Legendrian submanifold $L\sub M$ there exists a characteristic for $\al$
\footnote{This means a leaf of the foliation determined by the integrable distribution $\ker\big(d\al:TM\to T^*M\big)$ on $M$, i.e., an \emph{unparametrized} Reeb trajectory.}
that intersects $L$ at least twice. To explain this terminology, note that parametrizing part of such a characteristic, we obtain a \emph{strict Reeb chord}, i.e., an integral curve of the Reeb vector field that starts and ends at different points of $L$. 

Such chords arise in classical mechanics as libration motions, i.e., oscillations of a mechanical system between two rest points, see \cite[p.~118]{Ci}. The present article is concerned with the following problem.
\begin{Pb}[strict chord problem] Find conditions on $(M,\al)$ under which it has the strict chord property.
\end{Pb}
In \cite[p.~11]{Ar} V.~I.~Arnol'd conjectured that for $n\geq2$ any contact form on $S^{2n-1}$ inducing the standard structure has the strict chord property. The main result of this article roughly is that this property holds for every contact form on a manifold if all its characteristics are closed and of the same period and the contact form nicely embeds into the product of $\R^{2n}$ and an exact symplectic manifold. In particular, this confirms Arnol'd's conjecture for the standard form on $S^{2n-1}$.

To state the result, let $M$ be a manifold and $\al$ a contact form on $M$. The period of a closed characteristic $C$ for $\al$ is the number
\[\left|\int_C\iota^*\al\right|,\]
where $\iota:C\to M$ denotes the inclusion and we equip $C$ with either orientation. We denote by $q^1,p_1,\ldots,q^n,p_n$ the standard coordinates of $\R^{2n}$, and define
\[\lam_0:=\frac12\sum_{i=1}^n\big(q^idp_i-p_idq^i\big).\]
\begin{thm}[strict chord property]\label{thm:chord} The pair $(M,\al)$ has the strict chord property if all characteristics for $\al$ are closed and of equal period $T$, and there exist a manifold $W$ together with a one-form $\lam$, an integer $n\geq\frac12\dim W+2$, and an embedding $\phi:M\to\R^{2n}\x W$, such that $d\lam$ is a geometrically bounded
\footnote{We call a symplectic form $\om$ on a manifold $W$ geometrically bounded iff there exists an $\om$-compatible almost complex structure $J$ on $W$ such that the metric $\om(\cdot,J\cdot)$ is complete with bounded sectional curvature and injectivity radius bounded away from 0. Examples are closed symplectic manifolds, cotangent bundles of closed manifolds, and symplectic vector spaces.}
symplectic form, and
\begin{eqnarray}\label{eq:dim M}&\dim M=2n+\dim W-1,&\\
\label{eq:phi M}&\phi(M)\sub\BAR B^{2n}(T)\x W,&\\
\label{eq:phi lam 0 lam}&\phi^*(\lam_0\oplus\lam)=\al.&
\end{eqnarray}
Here $\BAR B^{2n}(a)$ denotes the closed ball in $\R^{2n}$ of radius $\sqrt{a/\pi}$. 
\end{thm}
The proof of Theorem \ref{thm:chord} is based on a result by Yu.~Chekanov, which implies that the displacement energy of a closed Lagrangian submanifold in a geometrically bounded symplectic manifold is at least the minimal symplectic action of the Lagrangian. 

Assuming by contradiction that there is no strict Reeb chord, such a Lagrangian is constructed from the given Legendrian submanifold by moving it with the Reeb flow. This technique is a variation on the approach used by K.~Mohnke in \cite{Mo}. (In that article the Lagrangian was obtained by moving the Legendrian both with the Reeb flow and with the Liouville flow.)

A crucial ingredient of the proof is the fact that the displacement energy of a compact subset $X$ of the closed unit ball $\BAR B^{2n}$ is strictly less than $\pi$, provided that $X$ does not contain the unit sphere. (See Lemma \ref{le:X} below.)

Theorem \ref{thm:chord} has the following immediate application. We denote by $\iota:S^{2n-1}\to\R^{2n}$ the inclusion.
\begin{cor}[sphere]\label{cor:sphere} For $n\geq2$ the standard contact form $\al_0:=\iota^*\lam_0$ on $S^{2n-1}$ has the strict chord property.
\end{cor}
More examples are obtained by the following construction. By an \emph{exact Hamiltonian $S^1$-manifold} we mean a triple consisting of a smooth manifold $W$, a smooth $S^1$-action $\rho$ on $W$, and an $\rho$-invariant one-form $\lam$ on $W$, such that $d\lam$ is non-degenerate. We fix such a triple and numbers $c\in(0,\infty)$ and $n\in\N\cup\{0\}$. We denote by $X$ the vector field generated by $\rho$.
\footnote{This is the infinitesimal action of the element $1\in\R=\Lie S^1$, where we identify $S^1$ with $\R/2\pi\Z$.}
 We define
\begin{eqnarray}\label{eq:H 0}&H_0:\R^{2n}\to\R,\quad H_0(x_0):=\frac12|x_0|^2,&\\
\label{eq:H}&H:=\lam(X):W\to\R,&\\
\label{eq:M}&M:=\big\{(x_0,x)\in\big(\BAR B^{2n}(2\pi c)\wo\{0\}\big)\x W\,\big|\,H_0(x_0)+H(x)=c\big\}.&
\end{eqnarray}
We denote by
\[\iota:M\to\wt W:=\R^{2n}\x W\]
the inclusion.
\begin{prop}[contact form]\label{prop:M al}The set $M$ is a (smoothly embedded) hypersurface in $\wt W$
\footnote{$M$ may have a boundary.}
, and
\begin{equation}\label{eq:al}\al:=\iota^*\big(\lam_0\oplus\lam\big)
\end{equation}
is a contact form on $M$ all of whose characteristics are closed and of period $2\pi c$.
\end{prop}
\begin{cor}[strict chord property]\label{cor:M al} If $(W,d\lam)$ is geometrically bounded and $n\geq\frac12\dim W+2$ then $(M,\al)$ has the strict chord property.
\end{cor}
\begin{proof}[Proof of Corollary \ref{cor:M al}] This follows from Theorem \ref{thm:chord}, using Proposition \ref{prop:M al} and the facts that $(\R^{2n},\om_0)$ is geometrically bounded, geometric boundedness is invariant under products, and that conditions (\ref{eq:dim M},\ref{eq:phi M},\ref{eq:phi lam 0 lam}) are satisfied with $\phi:=\iota$.
\end{proof}
\begin{xpl}\label{xpl:} Let $X$ be a manifold. We define $W:=T^*X$ and $\lam$ to be the canonical one-form on $W$. We fix a smooth $S^1$-action $\si$ on $X$ and define
\[\rho:S^1\to\Diff(W),\quad\rho(z)(q,p):=\big(\si_z(q),p\,d\si_z(q)^{-1}\big),\]
where $\si_z:=\si(z)$. The triple $(W,\rho,\lam)$ is an exact Hamiltonian $S^1$-manifold, and $(W,d\lam)$ is geometrically bounded. Hence by Corollary \ref{cor:M al} the pair $(M,\al)$, defined as in (\ref{eq:M},\ref{eq:al}), is a contact manifold that has the strict chord property, provided that $n\geq\frac12\dim W+2$.
\end{xpl}
\begin{Rmk} There exist contact forms on closed manifolds that do not have the strict chord property. The simplest example is the standard contact form on $S^1$.

Another example, which is taken from \cite{Mo}, goes as follows. We denote by $\ga$ the standard angular form on $S^1$ that integrates to $2\pi$. Consider the contact one-form on $M:=S^1\x S^2$ given by
\[\al:=x_1\ga+\frac12\big(x_2dx_3-x_3dx_2\big),\]
where $x\in S^2\sub\R^3$. Each Legendrian loop $S^1\x\big\{(0,x_2,x_3)\big\}$ intersects each Reeb orbit $\{z\}\x\{x_1=0\}$ (with $z\in S^1$) only once.
\end{Rmk}

\subsection{Minimal action of a regular coisotropic submanifold of complex projective space}
The bound on the displacement energy of a compact subset of $\BAR B^{2n}$, which is used in the proof of Theorem \ref{thm:chord}, and a coisotropic version of Chekanov's theorem have the following application. Let $(M,\om)$ be a symplectic manifold and $N\sub M$ a coisotropic submanifold.%
\footnote{This means that for every $x\in N$ the symplectic complement
\[T_xN^\om:=\big\{v\in T_xM\,\big|\,\om(v,w)=0,\,\forall w\in T_xN\big\}\]
is contained in $T_xN$.}

 We denote by $N_\om$ the set of all isotropic (or characteristic) leaves of $N$ %
\footnote{The set
\[TN^\om=\big\{(x,v)\,\big|\,x\in N,\,v\in T_xN^\om\big\}\sub TN\] 
is an involutive distribution on $N$. Hence by Frobenius' theorem it gives rise to a foliation on $N$. Its leaves are called the isotropic leaves of $N$.}%
, and by $\D\sub\R^2$ the closed unit disk. We define the \emph{action spectrum} and the \emph{minimal action (or area)} of $N$ to be
\begin{eqnarray}\label{eq:S N}&S(N):=S(M,N):=S(M,\om,N):=&\\
\nn&\left\{\displaystyle\int_\D u^*\om\,\bigg|\,u\in C^\infty(\D,M):\,\exists F\in N_\om:\, u(S^1)\sub F\right\},&\end{eqnarray}
\begin{eqnarray}\label{eq:A N}&A(N):=A(M,N):=A(M,\om,N):=&\\
\nn&\inf\big(S(N)\cap(0,\infty)\big)\in[0,\infty].&
\end{eqnarray}
\begin{Rmks} 
\begin{itemize}
\item In the case $N=M$ we have
\[A(M)=\inf\left(\left\{\int_{S^2}u^*\om\,\Big|\,u\in C^\infty(S^2,M)\right\}\cap(0,\infty)\right).\]
(See \cite[Lemma 29]{SZ}.)
\item If $N$ is Lagrangian then
\[A(N)=\inf\left(\left\{\int_\D u^*\om\,\bigg|\,u\in C^\infty(\D,M):\,u(S^1)\sub N\right\}\cap(0,\infty)\right).\]
\end{itemize}
\end{Rmks}
We call $N$ \emph{regular} iff there exists a manifold structure%
\footnote{The induced topology is by definition Hausdorff and second countable.}%
on $N_\om$, such that the canonical projection $\pi:N\to N_\om$ is a smooth submersion.%
\footnote{In this case the symplectic quotient of $N$ is well-defined in the sense that this manifold structure is unique and there exists a unique symplectic form on $N_\om$ that pulls back to $\iota^*\om$ under $\pi$. Here $\iota:N\to M$ denotes the inclusion. The manifold structure on $N_\om$ induces the quotient topology on this set.}%
 Examples are Lagrangian submanifolds, $N=M$, and the sphere $N=S^{2n-1}\sub M=\R^{2n}$. (For further examples see \cite{Zi}.) Regularity is invariant under taking products.

Let $n\in\N=\{1,2,\ldots\}$. We equip the complex projective space $\CP^n$ with the Fubini-Study form $\om_\FS$.%
\footnote{This form is normalized in such a way that the area of a projective line is $\pi$.}%
\begin{thm}[minimal action]\label{thm:action} Let $(M,\om)$ be a geometrically bounded symplectic manifold, and $\emptyset\neq N$ a regular closed coisotropic submanifold of $\CP^n\x M$ of dimension less than $2n$. Then 
\[A(N)\leq\frac\pi2.\]
\end{thm}
\begin{Rmks} 
\begin{itemize}
\item The hypothesis that $\dim N<2n$ cannot be dropped. Otherwise,
\[M:=\{\pt\},\quad N:=\CP^n,\quad\textrm{and}\]
\[M:=\CP^n,\quad\om:=-\om_\FS,\quad N:=\big\{(x,x)\,\big|\,x\in\CP^n\big\}\]
are counterexamples.
\item The hypothesis that $N$ be regular cannot be dropped. To see this, let $n\geq2$. Then there exists a closed hypersurface $N_0\sub\R^{2n}$ without any closed characteristic, see \cite{GG} and references therein.%
\footnote{In the case $n\geq3$ this hypersurface may be chosen to be smooth, but for $n=2$ the hypersurface constructed in \cite{GG} is only $C^2$.}%
 By shrinking $N_0$ with a homothety and using a Darboux chart, we otain a hypersurface $N$ inside $\CP^n$ with the same property. It satisfies $A(N)=\pi$ (but is not regular).
\end{itemize}
\end{Rmks}
\begin{cor}[minimal action]\label{cor:action} Let $(M,\om)$ be a geometrically bounded symplectic manifold, such that $\dim M<2n$. Then the minimal action of a closed nonempty Lagrangian submanifold of $\CP^n\x M$ is bounded above by $\pi/2$.
\end{cor}
\begin{proof} This follows from Theorem \ref{thm:action} and the fact that every Lagrangian submanifold is regular.
\end{proof}
\begin{Rmk} By Corollary \ref{cor:action} there is no \emph{exact} Lagrangian submanifold of $\CP^n\x M$. By this we mean a Lagrangian submanifold with minimal action equal to $\pi$. 
\end{Rmk}
To explain a further application of Theorem \ref{thm:action}, recall that a presymplectic form on a manifold $M$ is a closed two-form $\om$ on $M$, such that 
\[\corank\om_x:=\dim(T_xM)^\om\]
does not depend on $x\in M$.%
\footnote{Here
\[(T_xM)^\om=\big\{v\in T_xM\,\big|\,\om(v,w)=0,\,\forall w\in T_xM\big\}.\]}%
 Let $\om$ be such a form. The set
\[TM^\om=\big\{(x,v)\,\big|\,x\in N,\,v\in T_xM^\om\big\}\sub TM\] 
is an involutive distribution on $M$. Hence by Frobenius' theorem it gives rise to a foliation on $M$. We denote by $M_\om$ the set of its leaves. We call $(M,\om)$ regular iff there exists a manifold structure on $M_\om$ for which the canonical projection $\pi:M\to M_\om$ is a smooth submersion.

A presymplectic embedding of a presymplectic manifold into another one is by definition a smooth embedding that intertwines the two presymplectic forms.
\begin{cor}[presymplectic embedding]\label{cor:presympl} Let $n\in\N$ and $(M,\om)$ be a geometrically bounded symplectic manifold, such that 
\begin{equation}\label{eq:int S 2}\int_{S^2}u^*\om\in\pi\Z,\quad\forall u\in C^\infty(S^2,M).
\end{equation}
Let $(M',\om')$ be a nonempty closed regular presymplectic manifold, such that every isotropic leaf of $M'$ is simply connected, and 
\begin{eqnarray}\label{eq:dim M' corank}&\dim M'+\corank\om'=2n+\dim M,&\\
\nn&\dim M'<2n.&
\end{eqnarray}
Then $(M',\om')$ does not presymplectically embed into the symplectic manifold $\big(\CP^n\x M,\om_\FS\oplus\om\big)$. 
\end{cor}
\begin{Xpl} Let $F$ be a simply-connected closed manifold of positive dimension and $(X,\si)$ a closed symplectic manifold. By Corollary \ref{cor:presympl} the presymplectic manifold $\big(X\x F,\si\oplus0\big)$ does not embed into $\big(\CP^n,\om_\FS\big)$, where $n:=\frac12\dim X+\dim F$.
\end{Xpl}
Corollary \ref{cor:presympl} will be proved in Section \ref{sec:proof:thm:action}. It has the following immediate application.
\begin{cor}[Lagrangian embedding]\label{cor:simply} Let $n\in\N$ and $(M,\om)$ be a geometrically bounded symplectic manifold such that (\ref{eq:int S 2}) holds and $\dim M<2n$. Then no simply-connected closed manifold embeds into $\CP^n\x M$ in a Lagrangian way.
\end{cor}

\subsection{Related work}
V.~I.~Arnol'd observed in \cite{Ar} that the strict chord property for $(S^3,\al_0)$ follows from an elementary argument. In \cite[Corollary 1]{Giv} A.~B.~Givental' proved that there exists a Reeb chord between every pair of Legendrian submanifolds of $\RP^{2n-1}$ with the standard contact form, if they are isotopic via Legendrian submanifolds to the standard $\RP^{n-1}$. 

In \cite{ChCrit} Yu.~Chekanov provided lower bounds on the number of critical points of a quasi-function, i.e., a Legendrian submanifold of the 1-jet bundle of a manifold that is smoothly homotopic (via Legendrians) to the zero section. These points correspond to Reeb chords between the zero section and the Legendrian.

C.~Abbas \cite{AbNote,AbFinite,AbChord} proved the strict chord property for certain Legendrian knots in tight closed contact 3-manifolds. 

We say that a contact form $\al$ on a manifold $M$ has the \emph{chord property} iff every closed Legendrian intersects some characteristic for $\al$ at least twice or it intersects some closed characteristic (i.e., periodic Reeb orbit). Note that this property is trivially satisfied iff all characteristics are closed.

Consider now a contact manifold $(M,\xi)$ that arises as the boundary of a compact Stein manifold, and $\al$ a contact form on $M$ inducing $\xi$. In \cite[Theorem 2]{Mo} K.~Mohnke proved that $\al$ has the chord property. It follows that a nonempty closed Legendrian submanifold of $M$ admits a \emph{strict} Reeb chord if it does not intersect any closed characteristic for $\al$. Intuitively such Legendrian submanifolds are generic, provided that $\dim M\geq3$ and that $\al$ has only countably many closed Reeb orbits.

In \cite{Ci} K.~Cieliebak proved that Legendrian spheres in the boundaries of certain subcritical Weinstein domains intersect some characteristic for $\al$ at least twice.

Let $U\sub\R^{2n}$ be a bounded star-shaped domain with smooth boundary and $\emptyset\neq L\sub\dd U$ a closed Legendrian submanifold of nonpositive curvature. The last condition means that $L$ that admits a Riemannian metric of nonpositive sectional curvature. In the recent preprint \cite{CM} K.~Cieliebak and K.~Mohnke proved that $L$ possesses a Reeb chord of length bounded above by the (toroidal) Lagrangian capacity of $U$, see \cite[Corollary 1.12]{CM}. Using \cite[Corollary 1.3]{CM}, they deduced that $L$ admits a Reeb chord of length bounded above by $\pi/n$, if $n\geq2$, and $U=B^{2n}_1$, i.e., $\dd U$ is the unit sphere.

As explained in \cite{CM} after Corollary 1.13, it follows that there exists no exact Lagrangian embedding into $\CP^n$ of a closed manifold $\emptyset\neq X$ of nonpositive curvature. (Corollary \ref{cor:action} is a stabilized version of this without the nonpositive curvature assumption.)

K.~Cieliebak and K.~Mohnke also proved that for $S\sub\R^{2n}$ sufficiently $C^1$-close to the unit sphere, every closed Legendrian submanifold $\emptyset\neq L\sub S$ of nonpositive curvature possesses a strict Reeb chord, see \cite[Corollary 1.15]{CM}.

A powerful tool for finding Reeb chords is Legendrian contact homology. Based on work by Y.~Eliashberg, A.~Givental, H.~Hofer \cite{EGH} and Yu.~Chekanov \cite{ChDiff}, this homology was developed by F.~Bourgeois, T.~Ekholm, J.~B.~Etnyre, J.~Sabloff, M.~Sullivan, and others, see \cite{EESu,EESa,BEE,Ek} and references therein.

Using embedded contact homology, M.~Hutchings and C.~Taubes \cite{HT1,HT2} proved that every contact form on a closed 3-manifold has the chord property. Further results are contained in \cite{SaIt,SaMorse,Rit,Me}.

In \cite[Theorem 3.1]{Se} P.~Seidel proved that if a closed manifold $X$ embeds into $\CP^n$ in a Lagrangian way then $H^1\big(X,\Z/(2n+2)\big)\neq0$. In particular, $X$ is not simply-connected. Corollary \ref{cor:presympl} extends the latter statement to presymplectic embeddings into $\CP^n\x M$.

P.~Biran and K.~Cieliebak \cite{BCSympl,BCLag} generalized Seidel's result in various ways. In the case $\int_{S^2}u^*\om=0$, for every $u\in C^\infty(S^2,M)$, Corollary \ref{cor:simply} follows from their results. Further references about results on the topology of Lagrangian embeddings are provided in \cite{BCSympl,BCLag}.

\subsection{Acknowledgments}
I would like to thank the anonymous referee for his/ her useful comments. I would also like to thank K.~Cieliebak and K.~Mohnke for making me aware of the application regarding the minimal action of a Lagrangian submanifold of $\CP^n$ (Corollary \ref{cor:action}).
\section{Proof of Theorem \ref{thm:chord} (strict chord property)}
The proof of Theorem \ref{thm:chord} is based on the following construction. Let $M,\al,T,W,\lam,n,\phi$ be as in the hypothesis of Theorem \ref{thm:chord}, and $L\sub M$ a nonempty closed Legendrian submanifold. We construct a Lagrangian immersion in $\R^{2n}\x W$ by flowing $L$ with the Reeb flow. It will follow from Theorem \ref{thm:L} and Lemma \ref{le:X} below that this immersion is not injective. This means that $L$ admits a strict Reeb chord. 

We identify
\[S^1\iso\R/T\Z,\]
and denote by $R$ the Reeb vector field on $M$ \wrt $\al$, and by
\begin{equation}\label{eq:psi S 1}\psi:S^1\x M\to M\end{equation}
its flow. This map is welldefined, since by hypothesis all Reeb orbits of $\al$ are closed and of period $T$. We write
\begin{equation}\label{eq:wt W}\wt W:=\R^{2n}\x W,\quad\wt\lam:=\lam_0\oplus\lam,\quad\wt\om:=d\wt\lam,\end{equation}
and denote by $\iota:S^1\x L\to S^1\x M$ the inclusion. We define
\begin{equation}\label{eq:f}f:=\phi\circ\psi\circ\iota:S^1\x L\to\wt W.
\end{equation}
\begin{lemma}\label{le:Lag}The map $f$ is a Lagrangian immersion \wrt $\wt\om$.
\end{lemma}
\begin{proof}[Proof of Lemma \ref{le:Lag}] Since $\psi$ is the flow of $R$, we have
\begin{equation}\label{eq:d psi T S 1}d\psi(z,x)\big(T_zS^1\x\{0\}\big)=\R R_{\psi(z,x)},\quad\forall(z,x)\in S^1\x M.\end{equation}
We show that \textbf{$f$ is an immersion}. Since $L$ is Legendrian, we have $TL\sub\ker\al$. Since the Reeb flow $\psi_z:=\psi(z,\cdot)$ preserves $\al$, it preserves $\ker\al$, for every $z\in S^1$. It follows that $d\psi\big(\{0\}\x TL\big)\sub\ker\al$. Let $(z,x)\in S^1\x L$. Using (\ref{eq:d psi T S 1}) and the fact $\al(R)\const1$, it follows that
\begin{equation}\label{eq:d psi 0 TL}d\psi(z,x)\big(\{0\}\x T_xL\big)\cap d\psi(z,x)\big(T_zS^1\x\{0\}\big)=\{0\}.\end{equation}
Since $\psi$ is a flow, $d\psi_z(x)$ is injective. It follows from (\ref{eq:d psi T S 1}) and the fact $R\neq0$ that $d\big(\psi(\cdot,x)\big)(z)$ is injective. Combining this with (\ref{eq:d psi 0 TL}), it follows that
\[d\psi(z,x):T_{(z,x)}(S^1\x L)\to T_{\psi(z,x)}M\]
is injective. Using (\ref{eq:f}) and that $\phi$ is an immersion, it follows that the same holds for $f$, as claimed.\\

We show that $f$ is \textbf{isotropic}. We define $\om:=d\al$. The equalities (\ref{eq:phi lam 0 lam},\ref{eq:f}) and $\wt\om=d\wt\lam$ imply that
\[f^*\wt\om=d\iota^*\psi^*\phi^*\wt\lam=d\iota^*\psi^*\al=\iota^*\psi^*\om.\]
Therefore it suffices to show that $\psi^*\om$ vanishes on pairs of vectors in $T(S^1\x L)$ (over the same point). To see this, note that for every $z\in S^1$ the Reeb flow $\psi_z:M\to M$ preserves $\om$, since it preserves $\al$. Since $L$ is Legendrian, it is isotropic \wrt $\om$. It follows that $\psi^*\om$ vanishes on pairs of vectors in $\{0\}\x TL$.

The equalities (\ref{eq:d psi T S 1}) and $\R R=\ker\om$ imply that $\psi^*\om$ vanishes on each pair of vectors in $T(S^1\x L)$, of which at least one lies in $TS^1\x\{0\}$. It follows that $\psi^*\om$ vanishes on all pairs of vectors in $T(S^1\x L)$. This shows that $f$ is isotropic.\\

Equality (\ref{eq:dim M}) implies that the domain of $f$ has dimension equal to $\frac12\dim\wt W$. It follows that $f$ is a Lagrangian immersion, as claimed. This proves Lemma \ref{le:Lag}.
\end{proof}
The proof that the map $f$ is not injective, is based on the next result, which is due to Yu.~Chekanov. Let $(M,\om)$ be a symplectic manifold. We denote by $\HH(M,\om)$ the set of all functions $H\in C^\infty\big([0,1]\x M,\R\big)$ whose Hamiltonian time $t$ flow $\phi_H^t\colon M\to M$ exists and is surjective, for every $t\in[0,1]$.%
\footnote{The time $t$ flow of a time-dependent vector field on a manifold $M$ is always an injective smooth immersion on its domain of definition. Hence if it is everywhere well-defined and surjective then it is a diffeomorphism of $M$.}%
 We define the \emph{Hofer norm} 
\[\Vert\cdot\Vert\colon\HH(M,\om)\to[0,\infty],\quad\Vert H\Vert:=\int_0^1\left(\sup_MH^t-\inf_MH^t\right)dt,\]
and the \emph{displacement energy} of a subset $X\sub M$ to be 
\begin{eqnarray*}&e(X):=e(M,X):=e(M,\om,X):=&\\
&\inf\big\{\Vert H\Vert\,\big|\,H\in\HH(M,\om)\colon\phi_H^1(X)\cap X=\emptyset\big\}.\footnotemark&
\end{eqnarray*}
\footnotetext{Alternatively, one can define a displacement energy, using only functions $H$ with compact support. However, it seems more natural to allow for all functions in $\HH(M,\om)$.}
Let $L\sub M$ be a Lagrangian submanifold. We denote by $\D\sub\R^2$ the closed unit disk. The \emph{minimal symplectic action (or area) of $L$} is defined to be
\begin{eqnarray*}&A(M,\om,L):=&\\
&\inf\left(\left\{\displaystyle\int_\D u^*\om\,\big|\,u\in C^\infty(\D,M):\,u(S^1)\sub L\right\}\cap(0,\infty)\right)\in[0,\infty].&\end{eqnarray*}
\begin{thm}[displacement energy of a Lagrangian]\label{thm:L} If $(M,\om)$ is geometrically bounded and $L$ is closed then 
\[e(L)\geq A(L).\]
\end{thm}
\begin{proof}[Proof]\setcounter{claim}{0} This follows from the Main Theorem in \cite{ChLag}.
\end{proof}
Another key ingredient in the proof that $f$ (defined as in (\ref{eq:f})) is not injective, is the following lemma. Let $n\in\N$. We denote by $\om_0$ the standard symplectic form on $\R^{2n}$.
\begin{lemma}[bound on displacement energy]\label{le:X} Let $a>0$ and $X$ be a compact subset of the closed ball $\BAR B^{2n}(a)$, which does not contain $S^{2n-1}(a)$, the sphere in $\R^{2n}$ of radius $\sqrt{a/\pi}$. Then 
\begin{equation}\label{eq:e X}e\big(\R^{2n},X\big)<a.\end{equation}
\end{lemma}
\begin{proof}[Proof of Lemma \ref{le:X}] \WLOG we may assume that $a=\pi$. Since $X$ does not contain $S^{2n-1}$, there exists an orthogonal linear symplectic map $\Psi\colon\R^{2n}\to\R^{2n}$, such that $(1,0,\ldots,0)\not\in\Psi(X)$. We denote
\[Y_c:=\big\{(q,p)\in\D\mid q\leq c\big\}.\]
Since $\Psi(X)$ is compact and contained in $\BAR B^{2n}$, there exists $c<1$, such that
\begin{equation}\label{eq:phi X}\Psi(X)\sub Y_c\x\R^{2n-2}.
\end{equation}
We have
\begin{align*}e\big(\R^{2n},X\big)&=e\big(\R^{2n},\Psi(X)\big)\\
&\leq e\big(\R^{2n},Y_c\x\R^{2n-2}\big)\\
&\leq e\big(\R^2,Y_c\big)\\
&=\area(Y_c)\\
&<\pi\\
&=a.
\end{align*}
The fourth step follows from a concrete construction of a Hamiltonian diffeomorphism that displaces $Y_c$ or from a Moser type argument. This proves (\ref{eq:e X}) and hence Lemma \ref{le:X}.
\end{proof}
\begin{proof}[Proof of Theorem \ref{thm:chord}] Let $M,\al,T,W,\lam,n,\phi$ be as in the hypothesis of this theorem, and $L\sub M$ a nonempty closed Legendrian submanifold. By hypothesis all characteristics for $\al$ are closed, i.e., all Reeb orbits are periodic. Furthermore, their periods are all equal to $T$. We identify $S^1\iso\R/T\Z$ and define $\psi,\wt W,\wt\lam,\wt\om,f$ as in (\ref{eq:psi S 1},\ref{eq:wt W},\ref{eq:f}).
\begin{claim}\label{claim:f} The map $f$ is not injective.
\end{claim}
\begin{proof}[Proof of Claim \ref{claim:f}] We denote 
\[\wt L:=f\big(S^1\x L\big),\]
and by
\[\pr_1:\wt W\to\R^{2n}\]
the projection onto the first factor. The hypotheses (\ref{eq:dim M}) and $n\geq\frac12\dim W+2$ imply that
\[\dim(S^1\x L)=1+\frac{\dim M-1}2=n+\frac12\dim W\leq2n-2.\]
Hence by Sard's theorem, it follows that 
\[S^{2n-1}(T)\not\sub\pr_1\circ f(S^1\x L)=\pr_1(\wt L).\]
On the other hand, hypothesis (\ref{eq:phi M}) implies that
\[\pr_1(\wt L)\sub\BAR B^{2n}(T).\]
Therefore, applying Lemma \ref{le:X}, we have
\begin{equation}\label{eq:e T al}e\big(\wt W,\wt L\big)\leq e\big(\R^{2n},\pr_1(\wt L)\big)<T.
\end{equation}
\textbf{Assume now by contradiction that $f$ was injective.} This map is proper, since its domain is compact. Hence it follows from Lemma \ref{le:Lag} that $f$ is a Lagrangian embedding. Since $(\R^{2n},\om_0)$ and $(W,\om)$ are geometrically bounded, the same holds for $(\wt W,\wt\om)$. Therefore, Theorem \ref{thm:L} implies that
\[e\big(\wt W,\wt L\big)\geq A\big(\wt W,\wt L\big).\]
Combining this inequality with (\ref{eq:e T al}) and the next claim, we arrive at a contradiction.
\begin{claim}\label{claim:A} We have
\begin{equation}\label{eq:A T}A\big(\wt W,\wt L\big)\geq T.
\end{equation}
\end{claim}
\begin{pf}[Proof of Claim \ref{claim:A}] Let $\wt u\in C^\infty(\D,\wt W)$ be such that
\[\wt u(S^1=\dd\D)\sub\wt L=f\big((\R/T\Z)\x L\big).\]\
We show that 
\begin{equation}\label{eq:int D}\int_\D\wt u^*\wt\om\in T\Z.
\end{equation}
We define
\[x:=\phi^{-1}\circ\wt u:S^1\to\psi\big((\R/T\Z)\x L\big)\sub M.\]
(Recall that $\phi:M\to\wt W$ is the given embedding.) The equality $\wt\om=d\wt\lam$, Stokes' Theorem, and the hypothesis (\ref{eq:phi lam 0 lam}) imply that
\begin{equation}\label{eq:int D S 1}\int_\D\wt u^*\wt\om=\int_{S^1}(\phi\circ x)^*\wt\lam=\int_{S^1}x^*\al.\end{equation}
We define
\[(z,y):=\psi^{-1}\circ x:S^1\to(\R/T\Z)\x L.\]
This makes sense, since the restriction of $\psi$ to $(\R/T\Z)\x L$ is injective. (Here we use our assumption that $f$ is injective.) Since $x=\psi\circ(z,y)$, we have
\[x^*\al=\al\Big(R\circ\psi\circ(z,y)dz+d\psi_z(y)dy\Big)=dz+\al dy=dz.\]
Here we view $dz$ as a real-valued one-form on $S^1$. In the second equality we used that $\al(R)\const1$ and that the Reeb flow $\psi$ preserves $\al$. In the last equality we used that $L$ is Legendrian, and hence $\al\big|_{TL}=0$. It follows that
\[\int_{S^1}x^*\al=\int_{S^1}dz=T\deg(z).\]
Using (\ref{eq:int D S 1}), this proves (\ref{eq:A T}), i.e., Claim \ref{claim:A}, 
\end{pf}
and therefore Claim \ref{claim:f}.
\end{proof}
By Claim \ref{claim:f} there exist distinct points $(z_i,x_i)\in S^1\x L=(\R/T\Z)\x L$, $i=0,1$, such that
\begin{equation}\label{eq:f z 0}f(z_0,x_0)=f(z_1,x_1).\end{equation}
Recalling the definition (\ref{eq:f}), our hypothesis that the period of every characteristic equals $T$, implies that the map $f(\cdot,x_0)$ is injective. It follows that $x_0\neq x_1$. Using (\ref{eq:f},\ref{eq:f z 0}), it follows that these two points lie on the same characteristic for $\al$. Hence $L$ intersects this characteristic at least twice. This proves Theorem \ref{thm:chord}.
\end{proof}
\begin{Rmk}The above proof relies on the sharp bound for the displacement energy of a closed Lagrangian submanifold due to Yu.~Chekanov \cite{ChLag}. The same result was used by K.~Mohnke \cite{Mo} and later by K.~Cieliebak and K.~Mohnke \cite[Corollaries 1.4 and 1.5]{CM} to find Reeb chords. The construction of the closed Lagrangian submanifold in the proof of Theorem \ref{thm:chord} is a variation on the construction in \cite{Mo,CM}.

A new feature is that here the Reeb flow alone is used to produce a Lagrangian submanifold, whereas in \cite{Mo,CM} both the Reeb flow and the Liouville flow are used. The new approach works because of the upper bound on the displacement energy of a compact subset of a ball given in Lemma \ref{le:X}.
\end{Rmk}
\section{Proof of Proposition \ref{prop:M al} (contact form)}
The proof of Proposition \ref{prop:M al} is based on the following result. Let $(W,\rho,\lam)$ be an exact Hamiltonian $S^1$-manifold, and $c\in\R\wo\{0\}$. We denote by $X$ the vector field generated by $\rho$ and define
\[H:=\lam(X):W\to\R,\quad M:=H^{-1}(c)\sub W.\]
We denote by
\[\iota:M\to W\]
the inclusion.
\begin{prop}\label{prop:Ham} The set $M$ is a hypersurface in $W$, $\al:=\iota^*\lam$ is a contact form on $M$, and all characteristics of $\al$ are closed. Their periods are equal to $2\pi c$ if the restriction of the action $\rho$ to $M$ is free.
\end{prop}
\begin{proof}[Proof of Proposition \ref{prop:Ham}]\setcounter{claim}{0} By hypothesis the form 
\[\om:=d\lam\]
is nondegenerate, i.e., symplectic. We denote by $V$ the Liouville vector field on $W$ \wrt $\lam$. This is the unique vector field satisfying
\[\iota_V\om=\lam.\]
We have
\begin{align*}dH\cdot V&=\iota_Vd\iota_X\lam\\
&=\iota_V\LL_X\lam-\iota_V\iota_Xd\lam\\
&=0+\iota_X\iota_V\om\\
&=\iota_X\lam\\
&=H.
\end{align*}
Here in the second line we used Cartan's formula, and in the third line we used our hypothesis that $\lam$ is $\rho$-invariant. It follows that $dH\cdot V\const c\neq0$ along $M=H^{-1}(c)$. Hence $c$ is a regular value for $H$, $M$ is a hypersurface in $W$, and the Liouville vector field $V$ is transverse to $M$. It follows that $\al:=\iota^*\lam$ is a contact form on $M$. By the next claim its characteristics are closed.

\begin{claim}\label{claim:char}The characteristics of $\al$ are the orbits of the restriction of $\rho$ to $M$. 
\end{claim}
\begin{proof}[Proof of Claim \ref{claim:char}] It suffices to show that $X$ is $c$ times the Reeb vector field of $\al$. To see this, note that $X$ is tangent to $M$, since
\[dH\cdot X=\iota_Xd\iota_X\lam=\iota_X\LL_X\lam-\iota_X\iota_Xd\lam=0-0.\]
By definition, we have
\[\al(X)=\lam(X)\const c\textrm{ on }M=H^{-1}(c).\]
Finally,
\[\iota_Xd\al=\iota_Xd\lam=\LL_X\lam-d\iota_X\lam=0-dH=0\textrm{ on }TM.\]
It follows that $X$ equals $c$ times the Reeb vector field of $\al$. This proves Claim \ref{claim:char}.
\end{proof}

Assume now that the restriction of $\rho$ to $M$ is free. Let $C$ be a characteristic for $\al$. We choose $x_0\in C$ and define
\begin{equation}\label{eq:phi S1}\phi:S^1\to C,\quad\phi(z):=\rho(z,x_0).\end{equation}
This is a diffeomorphism, since the restriction of $\rho$ to $C$ is free. We denote by $\iota:C\to M$ the inclusion and by $\ga$ the standard angular one-form on $S^1$, whose integral equals $2\pi$. We have
\[\phi^*\iota^*\al=(\iota\circ\phi)^*\al=\big(\lam(X)\circ\phi\big)\ga=(H\circ\phi)\ga=c\ga.\]
Here in the second step we used the fact that $X$ generates the action $\rho$, and (\ref{eq:phi S1}). It follows that
\[\int_C\iota^*\al=\int_{S^1}\phi^*\iota^*\al=2\pi c.\]
Here we equipped $C$ with the orientation induced by the standard orientation on $S^1$ and the map $\phi$. This proves Proposition \ref{prop:Ham}.
\end{proof}
\begin{proof}[Proof of Proposition \ref{prop:M al}] We denote by $\rho_0$ the standard diagonal $S^1$-action on $\R^{2n}=\C^n$, given by
\[\rho_0(z)z_0:=zz_0=\big(zz_0^1,\ldots,zz_0^n\big).\]
By $\rho_0\x\rho$ we denote the product $S^1$-action on $\R^{2n}\x W$. We define $H_0,H$ as in (\ref{eq:H 0},\ref{eq:H}). The triple
\[\big(\wt W,\wt\rho,\wt\lam\big):=\Big(\big(\BAR B^{2n}(2\pi c)\wo\{0\}\big)\x W,\big(\rho_0\x\rho\big)|_{\wt W},\lam_0\oplus\lam\Big)\]
is an exact Hamiltonian $S^1$-action, and 
\[\wt H:=H_0\oplus H=\iota_{\wt X}\wt\lam:\wt W\to\R,\]
where $\wt X$ denotes the vector field generated by $\wt\rho$. The set $M$ defined in (\ref{eq:M}), is given by
\[M=\wt H^{-1}(c).\]
Since the restriction of $\rho_0$ to $\BAR B^{2n}(2\pi c)\wo\{0\}$ is free, the action $\wt\rho$ is free. Therefore, by Proposition \ref{prop:Ham} $M$ is a hypersurface in $\wt W$, and $\al:=\iota^*\wt\lam$ is a contact form on $M$ all of whose characteristics are closed and of period $2\pi c$. This proves Proposition \ref{prop:M al}.
\end{proof}
\section{Proof of Theorem \ref{thm:action} (minimal action) and of Corollary \ref{cor:presympl} (presymplectic embedding)}\label{sec:proof:thm:action}
In this section we denote by
\[B^n_r,\quad\BAR B^n_r,\quad S^{n-1}_r\]
the open and closed balls around 0 in $\R^n$ of radius $r$, and the sphere around 0 in $\R^n$ of radius $r$.\\

The proof of Theorem \ref{thm:action} is based on Lemma \ref{le:X} (bound on displacement energy) and the following. Let $(M,\om)$ be a symplectic manifold.
\begin{thm}\label{thm:leafwise} Assume that $(M,\om)$ is geometrically bounded. Let $N\sub M$ be a closed, regular coisotropic submanifold. Then 
\[e(N)\geq A(N).\]
\end{thm}
\begin{proof} This is an immediate consequence of \cite[Theorem 1.1]{Zi}.
\end{proof}
\begin{Rmk} 
This theorem generalizes Chekanov's Theorem \ref{thm:L}.
\end{Rmk}
In the proof of Theorem \ref{thm:action} we will also use the following lemma. Let $(M,\om)$ be a presymplectic manifold. Recall that this means that $\om$ is closed two-form on $M$, and $\corank\om_x:=\dim(T_xM)^\om$ does not depend on $x\in M$. Let $N\sub M$ be a coisotropic submanifold. This means that for every $x\in N$ the space $(T_xN)^\om$ is contained in $T_xN$. We denote by $\iota:N\to M$ the inclusion. 
\begin{rmk}\label{rmk:presympl} The form $\iota^*\om$ is presymplectic. That its corank is constant, follows from Lemma \ref{le:W coiso} and Remark \ref{rmk:W i om} in the appendix.
\end{rmk}
By Remark \ref{rmk:presympl} the distribution $(TN)^\om$ defines a foliation on $N$. We denote by $N_\om$ the set of its leaves and define the \emph{action spectrum} $S(N)=S(M,N)=S(M,\om,N)$ and the \emph{minimal action (or area)}
\[A(N)=A(M,N)=A(M,\om,N)\]
of such a submanifold as in (\ref{eq:S N},\ref{eq:A N}).
\begin{lemma}[lift of coisotropic submanifold]\label{le:lift} Let $(M,\om)$ and $(M',\om')$ be presymplectic manifolds, $f:M'\to M$ a surjective proper presymplectic%
\footnote{This means that $f^*\om=\om'$.}%
 submersion, and $N\sub M$ a coisotropic submanifold. Then the following statements hold:
\begin{enui}
\item\label{le:lift:coiso} The set $N':=f^{-1}(N)$ is a coisotropic submanifold of $M'$.
\item\label{le:lift:action} 
\begin{equation}\label{eq:A om wt N}A\big(M,N\big)\leq A(M',N').
\end{equation}
\item\label{le:lift:reg} Assume that $N$ is regular and, for all $x',y'\in M'$,
\begin{equation}\label{eq:f x' f y'}f(x')=f(y')\then x'\textrm{ and }y'\textrm{ lie on the same isotropic leaf of }M'.
\end{equation}
Then $N'$ is regular.
\end{enui}
\end{lemma}
\begin{Rmk} In fact equality in (\ref{eq:A om wt N}) holds. However, this will not be used here.
\end{Rmk}
In the proof of Lemma \ref{le:lift} we will use the following. By a presymplectic vector space we mean a vector space together with a skew-symmetric bilinear form. 
\begin{lemma}\label{le:V V'} Let $(V,\om)$ and $(V',\om')$ be presymplectic vector spaces, $\Phi:V'\to V$ a linear presymplectic map%
\footnote{This means that $\Phi^*\om=\om'$.}%
, and $W\sub V$ a linear subspace. Then the following statements hold:
\begin{enui}
\item\label{le:V V':Phi -1 W} 
\begin{equation}\label{eq:Phi -1 W}\Phi^{-1}(W^\om)\sub\big(\Phi^{-1}(W)\big)^{\om'}.\end{equation}
\item\label{le:V V':surj} If $\Phi$ is surjective then the inclusion ``$\cont$'' in (\ref{eq:Phi -1 W}) holds.
\end{enui}
\end{lemma}
\begin{proof}[Proof of Lemma \ref{le:V V'}] This follows from the definitions.
\end{proof}
The proof of Lemma \ref{le:lift}(\ref{le:lift:reg}) is based on the following. Let $M$ be a (smooth finite-dimensional) manifold and $\F$ a (smooth) foliation on $M$, i.e., a maximal atlas of foliation charts. We denote by $R^\F$ its \emph{leaf relation}. This is the subset of $M\x M$ consisting of pairs of points lying in the same leaf. We call $\F$ \emph{regular} iff there exists a manifold structure%
\footnote{The induced topology is by definition Hausdorff and second countable.}%
 on the set of leaves $M/R^\F$, such that the canonical projection $\pi^\F:M\to M/R^\F$ is a (smooth) submersion.
\begin{lemma}\label{le:reg} Let $(M,\F)$ and $(M',\F')$ be foliated manifolds, such that $\F$ is regular. Let $f:M'\to M$ be a smooth surjective submersion such that
\begin{equation}\label{eq:x' R F y'}x'R^{\F'} y'\iff f(x')R^{\F}f(y'),\quad\forall x',y'\in M'.
\end{equation}
Then $\F'$ is regular.
\end{lemma}
\begin{proof}[Proof of Lemma \ref{le:reg}] We define the map
\[\phi:M'/R^{\F'}\to M/R^\F,\quad\phi(F'):=[f(x')],\]
where $x'\in F'$ is an arbitrary point. It follows from (\ref{eq:x' R F y'}) that this map is well-defined and injective. Our hypothesis that $f$ is surjective implies that $\phi$ is surjective, as well. By our assumption that $\F$ is regular there exists a manifold structure $\A$ on $M/R^\F$, for which the canonical projection $\pi^\F:M\to M/R^\F$ is a smooth submersion. Since $f$ is a smooth submersion and
\[\pi^{\F'}=\phi^{-1}\circ\pi^\F\circ f,\]
the map $\pi^{\F'}$ is a smooth submersion \wrt the pullback of $\A$ under $\phi$. Hence $\F'$ is regular. This proves Lemma \ref{le:reg}.
\end{proof}
\begin{proof}[Proof of Lemma \ref{le:lift}]\setcounter{claim}{0} \textbf{(\ref{le:lift:coiso}):} Since $f$ is a submersion, $N'$ is a submanifold of $M'$. It follows from Lemma \ref{le:V V'}(\ref{le:V V':surj}) that it is coisotropic. This proves (\ref{le:lift:coiso}).\\

To prove (\ref{le:lift:action},\ref{le:lift:reg}), we denote by $R^{N,\om}$ the isotropic leaf relation on $N$. This is the subset of $N\x N$ consisting of pairs of points that lie in the same isotropic leaf of $N$.
\begin{claim}\label{claim:x' y'}
\begin{enua}
\item\label{claim:x' y':then} If $(x'_0,x'_1)\in R^{N',\om'}$ then $\big(f(x'_0),f(x'_1)\big)\in R^{N,\om}$.
\item\label{claim:x' y':follows} If $x'_0,x'_1\in N'$ are such that $\big(f(x'_0),f(x'_1)\big)\in R^{N,\om}$ then
\begin{equation}\label{eq:N' om'}N'_{x'_0}\cap f^{-1}(f(x'_1))\neq\emptyset.\end{equation}
Here $N'_{x'_0}$ denotes the isotropic leaf of $N'$ through $x'_0$.
\end{enua}
\end{claim}
\begin{proof}[Proof of Claim \ref{claim:x' y'}] Let $x'\in N'$. Since $f$ is a submersion, we have
\[T_{x'}N'=df(x')^{-1}(T_{f(x')}N).\]
Using that $f$ is presymplectic, Lemma \ref{le:V V'} therefore implies that
\begin{equation}\label{eq:T x' N'}(T_{x'}N')^{\om'}=df(x')^{-1}(T_{f(x')}N)^\om.\end{equation}
It follows that $f\big(N'_{x'}\big)\sub N_{f(x')}$. This proves \textbf{(\ref{claim:x' y':then})}.\\

\textbf{Proof of (\ref{claim:x' y':follows}):} We choose a path $x\in C^\infty\big([0,1],N\big)$ that is tangent to $(TN)^\om$ and satisfies $x(i)=f(x'_i)$ for $i=0,1$. Since $f$ is a proper submersion, by Proposition \ref{prop:lift path} in the appendix there exists a path $x'\in C^\infty\big([0,1],M'\big)$ satisfying $x'(0)=x'_0$ and $f\circ x'=x$. It follows that $x'([0,1])\sub N'$. Since $\dot x(t)\in(T_{x(t)}N)^\om$, equality (\ref{eq:T x' N'}) implies that $\dot x'(t)\in(T_{x'(t)}N')^{\om'}$, for every $t\in[0,1]$. It follows that $x'(1)\in N'_{x'(0)}$. Since $x'(0)=x'_0$ and $f(x'(1))=x(1)=f(x'_1)$, condition (\ref{eq:N' om'}) follows. This proves (\ref{claim:x' y':follows}) and completes the proof of Claim \ref{claim:x' y'}.
\end{proof}

\textbf{Proof of (\ref{le:lift:action}):} Let $u'\in C^\infty(\D,M')$ be such that $u'(S^1)$ is contained in some isotropic leaf of $N'$. Claim \ref{claim:x' y'}(\ref{claim:x' y':then}) implies that $f\circ u'(S^1)$ is contained in some isotropic leaf of $N$. Since $f$ is presymplectic, we have 
\[\int_\D{u'}^*\om'=\int_\D(f\circ u')^*\om.\]
It follows that $S(M',N')\sub S(M,N)$, and therefore,
\[A(M',N')\geq A(M,N).\] 
This proves (\ref{le:lift:action}).\\

\textbf{Proof of (\ref{le:lift:reg}):} By Claim \ref{claim:x' y'}(\ref{claim:x' y':then}) the implication ``$\then$'' in condition (\ref{eq:x' R F y'}) with $M,M'$ replaced by $N,N'$, and $\F=\F^{N,\om},$ $\F'=\F^{N',\om'}$, is satisfied. Here $\F^{N,\om}$ denotes the isotropic foliation on $N$ \wrt $\om$.\\

To see the opposite implication, let $x'_0,x'_1\in N'$ be such that the relation $f(x'_0)R^{N,\om}f(x'_1)$ holds. By Claim \ref{claim:x' y'}(\ref{claim:x' y':follows}) there exists $y'_1\in N'_{x'_0}\cap f^{-1}(f(x'_1))$. Since $f(x'_1)=f(y'_1)$, our hypothesis (\ref{eq:f x' f y'}) implies that 
\[(x'_1,y'_1)\in R^{M',\om'}\cap(N'\x N')\sub R^{N',\om'}.\]
Since $(x'_0,y'_1)\in R^{N',\om'}$, it follows that $(x'_0,x'_1)\in R^{N',\om'}$. This shows the implication ``$\follows$'' in (\ref{eq:x' R F y'}) with $M,M'$ replaced by $N,N'$, and $\F=\F^{N,\om}$, $\F'=\F^{N',\om'}$. Hence (\ref{eq:x' R F y'}) is satisfied. Therefore, applying Lemma \ref{le:reg}, it follows that $N'$ is regular. This proves (\ref{le:lift:reg}) and completes the proof of Lemma \ref{le:lift}.
\end{proof}
In the proof of Theorem \ref{thm:action} we will also use the following lemma. 
\begin{lemma}\label{le:M' M} Let $(M,\om)$ be a presymplectic manifold, $M'$ a coisotropic submanifold of $M$ and $M''$ a coisotropic submanifold of $M'$. Then the following holds.
\begin{enui}
\item\label{le:M' M:coiso} $M''$ is a coisotropic submanifold of $M$.
\item\label{le:M' M:retract} If $M$ strongly smoothly deformation retracts onto $M'$ then
\begin{equation}\label{eq:A M om N} A\big(M',M''\big)\leq A(M,M'').
\end{equation}
\end{enui}
\end{lemma}
\begin{rmks}\label{rmk:M' M}
\begin{enui}\item\label{rmk:M' M:retract} That $M$ strongly smoothly deformation retracts onto $M'$ means that there exists a smooth map $h:[0,1]\x M\to M$ such that
\[h(0,\cdot)=\id,\quad h\big(\{1\}\x M\big)\sub M',\quad h(t,x)=x,\,\forall t\in[0,1],\,x\in M'.\]
\item The inequality ``$\geq$'' in (\ref{eq:A M om N}) is true without the retraction condition. However, this will not be used here.
\end{enui}
\end{rmks}
In the proof of Lemma \ref{le:M' M} we will use the following lemma.
\begin{lemma}\label{le:V V' V''} Let $(V,\om)$ be a finite-dimensional presymplectic vector space, $V'$ a coisotropic subspace of $(V,\om)$, and $V''$ a coisotropic subspace of $\big(V',\om':=\om|_{V'\x V'}\big)$. Then $V''$ is a coisotropic subspace of $(V,\om)$.
\end{lemma}
\begin{proof}[Proof of Lemma \ref{le:V V' V''}] This follows from Lemma \ref{le:W coiso} in the appendix.
\end{proof}
\begin{proof}[Proof of Lemma \ref{le:M' M}]\textbf{(\ref{le:M' M:coiso}):} This follows from Lemma \ref{le:V V' V''}.\\

\textbf{We prove (\ref{le:M' M:retract}).} It suffices to show that 
\begin{equation}\label{eq:S M M''}S(M,M'')\sub S(M',M'').
\end{equation}
Let $u\in C^\infty(\D,M)$ be such that $u(S^1)$ is contained in some isotropic leaf of $M''$. We choose a map $h$ as in Remark \ref{rmk:M' M}(\ref{rmk:M' M:retract}). We denote $h_t:=h(t,\cdot)$ and define 
\[f:[0,1]\x\D\to M,\quad f(t,z):=h_t\circ u(z).\]
Using that $d\om=0$ and Stokes' theorem%
\footnote{We use a version of this result that allows the manifold to have corners. See e.g.~\cite[Theorem 16.25]{Le}.}%
, we have
\begin{align}\nn0&=\int_{[0,1]\x\D}df^*\om\\
\nn&=\int_{\dd\big([0,1]\x\D\big)}f^*\om\\
\nn&=\int_\D(h_1\circ u)^*\om-\int_\D(h_0\circ u)^*\om+\int_{[0,1]\x S^1}f^*\om\\
\label{eq:int D h1}&=\int_\D(h_1\circ u)^*\om-\int_\D u^*\om+0.
\end{align}
Here in the last equality we used the fact $h_0=\id$ and that $u(S^1)\sub M''\sub M'$, and therefore $h_t\circ u|_{S^1}$ is constant in $t$. Since $h_1(M)=h\big(\{1\}\x M\big)\sub M'$, the map $h_1\circ u$ takes values in $M'$. It is of the sort occurring in the definition of $S(M',M'')$. Hence (\ref{eq:int D h1}) implies (\ref{eq:S M M''}). This proves (\ref{le:M' M:retract}) and completes the proof of Lemma \ref{le:M' M}.
\end{proof}
In the proof of Theorem \ref{thm:action} we will also use the following lemma. 
\begin{lemma}\label{le:M M'} Let $(M,\om)$ and $(M',\om')$ be presymplectic manifolds, $N\sub M\x M'$ a coisotropic submanifold, and $x\in M$, such that $\dim M>2$ and $N\cap\big(\{x\}\x M'\big)=\emptyset$. Then
\begin{equation}\label{eq:A M x}A\big((M\wo\{x\})\x M',N\big)\leq A\big(M\x M',N\big).
\end{equation}
\end{lemma}
\begin{Rmk} In fact equality in (\ref{eq:A M x}) holds. However, this will not be used here.
\end{Rmk}
\begin{proof}[Proof of Lemma \ref{le:M M'}] It suffices to prove that
\begin{equation}\label{eq:S M M'}S\big(M\x M',N\big)\sub S\big((M\wo\{x\})\x M',N\big).
\end{equation}
Let
\[\wt u=(u,u')\in C^\infty\big(\D,M\x M'\big)\]
be a map that sends $S^1$ to some isotropic leaf of $N$. Since $\dim M>2$, by Sard's Theorem $M\wo u(B^2_1)$ is dense in $M$. Hence an argument in a chart shows that there exists a smooth map $h:[0,1]\x M\to M$, such that
\begin{eqnarray*}&h(0,\cdot)=\id,\quad x\not\in h\big(\{1\}\x u(B^2_1)\big),&\\
&h(t,\cdot)=\id\textrm{ in some neighbourhood of }\pr_1(N)\sub M.&\end{eqnarray*}
Here we denoted by $\pr_1:M\x M'\to M$ the canonical projection, and we used the hypothesis that $N\cap\big(\{x\}\x M'\big)=\emptyset$. We denote $h_t:=h(t,\cdot)$ and define 
\[f:[0,1]\x\D\to M,\quad f(t,z):=h_t\circ u(z).\]
We have, as in (\ref{eq:int D h1}),
\[\int_\D(h_1\circ u)^*\om=\int_\D u^*\om.\]
Here we used the facts $h_0=\id$, $u(S^1)=\pr_1\circ\wt u(S^1)\sub\pr_1(N)$, and $h_t=\id$ in a neighbourhood of $\pr_1(N)$. It follows that
\begin{equation}\label{eq:int D h1 u u'}\int_\D\big(h_1\circ u,u'\big)^*\wt\om=\int_\D\wt u^*\wt\om.
\end{equation}
Since $x\not\in h_1(u(B^2_1))$, the map
\[\big(h_1\circ u,u'\big):\D\to\big(M\wo\{0\}\big)\x M'\]
is of the sort occurring in the definition of $S\big((M\wo\{x\})\x M',N\big)$. Hence (\ref{eq:int D h1 u u'}) implies (\ref{eq:S M M'}). This proves Lemma \ref{le:M M'}.
\end{proof}
In the proof of Theorem \ref{thm:action} we will also use the following.
\begin{lemma}\label{le:S N} Let $(M,\om)$ be a connected symplectic manifold and $N\sub M$ coisotropic submanifold. Then
\begin{equation}\label{eq:S N int}S(N)+\left\{\int_{S^2}u^*\om\,\Big|\,u\in C^\infty(S^2,M)\right\}\sub S(N).\end{equation}
\end{lemma}
\begin{proof}[Proof of Lemma \ref{le:S N}]\setcounter{claim}{0} Let $u\in C^\infty(\D,M)$ be such that $u(S^1)$ is contained in some isotropic leaf of $N$, and $v\in C^\infty(S^2,M)$. We choose a point $z_0\in S^2$.
\begin{claim}\label{claim:wt u wt v} There exist maps $\wt u\in C^\infty\big(\D\wo B^2_{\frac12},M\big)$ and $\wt v\in C^\infty(\BAR B^2_{\frac13},M)$ such that
\begin{equation}\label{eq:int D wt u}\int_{\D\wo B^2_{\frac12}}\wt u^*\om=\int_\D u^*\om,\quad\int_{\BAR B^2_{\frac13}}\wt v^*\om=\int_{S^2}v^*\om,\end{equation}
$\wt u=u$ in some neighbourhood of $S^1$, $\wt u\const u(0)$ in some neighbourhood of $S^1_{\frac12}$, and $\wt v\const v(z_0)$ in some neighbourhood of $S^1_{\frac13}$.
\end{claim}
\begin{proof}[Proof of Claim \ref{claim:wt u wt v}] We choose a map $f\in C^\infty\big(\D\wo B^2_{\frac12},\D\big)$ that restricts to an orientation preserving diffeomorphism from $\D\wo\BAR B^2_{\frac34}$ to $\D\wo\{0\}$, equals identity in a neighbourhood of $S^1$, and sends $\BAR B^2_{\frac34}\wo B^2_{\frac12}$ to $0$. 
We define
\[\wt u:=u\circ f:\D\wo B^2_{\frac12}\to M.\]
This map has the required properties.

To construct $\wt v$, we choose a map $g\in C^\infty\big(\BAR B^2_{\frac13},S^2)$ that restricts to an orientation preserving diffeomorphism from $B^2_{\frac14}$ to $S^2\wo\{z_0\}$ and sends $\BAR B^2_{\frac13}\wo B^2_{\frac14}$ to $z_0$. The map $\wt v:=v\circ g$ has the required properties. This proves Claim \ref{claim:wt u wt v}.
\end{proof}
We choose $\wt u,\wt v$ as in this claim. Since $M$ is connected, there exists a path $x\in C^\infty\left(\left[\frac13,\frac12\right],M\right)$, such that $x(0)=v(z_0)$ and $x(1)=u(0)$. We may modify $x$, such that it is constant in some neighbourhoods of $\frac13$ and $\frac12$. We define
\[w(z):=\left\{\begin{array}{ll}
\wt v(z),&\textrm{if }z\in \BAR B^2_{\frac13},\\
x(|z|),&\textrm{if }z\in B^2_{\frac12}\wo\BAR B^2_{\frac13},\\
\wt u(z),&\textrm{if }z\in\D\wo B^2_{\frac12}.
\end{array}\right.\] 
This map is smooth. It follows from (\ref{eq:int D wt u}) that
\begin{equation}\label{eq:int D w}\int_\D w^*\om=\int_{S^2}v^*\om+0+\int_\D u^*\om.\end{equation}
Since $w=u$ in some neighbourhood of $S^1$, the image $w(S^1)$ is contained in some isotropic leaf of $N$. It follows that $\int_\D w^*\om\in S(N)$. Combining this with (\ref{eq:int D w}), the inclusion (\ref{eq:S N int}) follows. This proves Lemma \ref{le:S N}.
\end{proof}
In the proof of Theorem \ref{thm:action} we will also use the following.
\begin{rmk}\label{rmk:-S N} Let $(M,\om)$ be a symplectic manifold and $N\sub M$ a coisotropic submanifold. Then
\[S(N)=-S(N)=\big\{-a\,\big|\,a\in S(N)\big\}.\]
This follows from the fact that for every $u\in C^\infty(\D,M)$ we have
\[\int_\D\BAR u^*\om=-\int_\D u^*\om,\]
where $\BAR u(z):=u(\BAR z)$, for every $z\in\D\sub\C$.
\end{rmk}
\begin{proof}[Proof of Theorem \ref{thm:action}] We denote by
\[\pi:S^{2n+1}\x M\to\CP^n\x M\]
the canonical projection, by $\iota:S^{2n+1}\to\R^{2n}$ the inclusion, and by $\om_0$ the standard symplectic form on $\R^{2n}$. We equip $S^{2n+1}\x M$ with the presymplectic form $\iota^*\om_0\oplus\om$. It follows from Lemma \ref{le:lift}(\ref{le:lift:coiso},\ref{le:lift:action}) that $N'=\pi^{-1}(N)$ is a coisotropic submanifold of $S^{2n+1}\x M$, and
\begin{equation}\label{eq:A CP n} A\big(\CP^n\x M,N\big)\leq A\big(S^{2n+1}\x M,N'\big).
\end{equation}
Since $\pi$ is proper and $N$ is compact, $N'$ is compact. The manifold $\big(\R^{2n+2}\wo\{0\}\big)\x M$ strongly smoothly deformation retracts onto $S^{2n+1}\x M$. Hence by Lemma \ref{le:M' M}(\ref{le:M' M:retract}), we have
\begin{equation}\label{eq:A S 2n 1}A\big(S^{2n+1}\x M,N'\big)\leq A\big(\big(\R^{2n+2}\wo\{0\}\big)\x M,N'\big).
\end{equation}
Since $n\geq1$, by Lemma \ref{le:M M'} we have
\begin{equation}\label{eq:A R 2n 2}A\big(\big(\R^{2n+2}\wo\{0\}\big)\x M,N'\big)\leq A\big(\R^{2n+2}\x M,N'\big).\end{equation}
The symplectic manifold $\R^{2n+2}$ is geometrically bounded. Using our hypothesis that $M$ is geometrically bounded, it follows that $\R^{2n+2}\x M$ has the same property. Since by hypothesis $N$ is regular, by Lemma \ref{le:lift}(\ref{le:lift:reg}) with $f=\pi$ the same holds for $N'$. (Condition (\ref{eq:f x' f y'}) with $M':=S^{2n+1}\x M$ is satisfied, since
\[R^{S^{2n+1}\x M}=\big\{\big((x,y),(zx,y)\big)\,\big|\,(x,y)\in S^{2n+1}\x M,\,z\in S^1\big\},\]
where we consider $S^{2n+1}$ as a subset of $\C^{n+1}$ and $S^1\sub\C$.) Hence applying Theorem \ref{thm:leafwise}, we obtain
\begin{equation}\label{eq:A e}A\big(\R^{2n+2}\x M,N'\big)\leq e\big(\R^{2n+2}\x M,N'\big).
\end{equation}
We denote by $\pr_1:S^{2n+1}\x M\to S^{2n+1}$ the projection onto the first factor. We have
\begin{align}\nn e\big(\R^{2n+2}\x M,N'\big)&\leq e\big(\R^{2n+2}\x M,\pr_1(N')\x M\big)\\
\label{eq:e R 2n 2}&\leq e\big(\R^{2n+2},\pr_1(N')\big).\end{align}
Our hypothesis $\dim M<2n$ implies that $\dim N'=\dim N+1\leq 2n$. Hence the restriction $\pr_1|_{N'}:N'\to S^{2n+1}$ is not submersive at any point, and therefore the set of its regular values is the complement of its image. Hence by Sard's Theorem $\pr_1(N')\neq S^{2n+1}$. Therefore by Lemma \ref{le:X} we have
\[e\big(\R^{2n+2},\pr_1(N')\big)<\pi.\]
Combining this with (\ref{eq:A CP n}
-\ref{eq:e R 2n 2}), it follows that 
\[A(N)=A\big(\CP^n\x M,N\big)<\pi.\]
Hence there exists $a\in S(N)\cap(0,\pi)$. If $a\leq\frac\pi2$ then it follows that $A(N)\leq\frac\pi2$, as claimed. Otherwise $-a+\pi<\frac\pi2$. By Remark \ref{rmk:-S N} we have $-a\in S(N)$. Since there exists $u\in C^\infty(S^2,\CP^n)$, such that $\int_{S^2}u^*\om_\FS=\pi$, Lemma \ref{le:S N} implies that $-a+\pi\in S(N)$. Since $-a+\pi<\frac\pi2$, it follows that $A(N)<\frac\pi2$. Hence in every case we have $A(N)\leq\frac\pi2$. This proves Theorem \ref{thm:action}.
\end{proof}
\begin{proof}[Proof of Corollary \ref{cor:presympl}] We denote
\[\wt M:=\CP^n\x M,\quad\wt\om:=\om_\FS\oplus\om.\]
Assume by contradiction that there exists a presymplectic embedding $\phi:M'\to\wt M$. We denote $N:=\phi(M')$. It follows from our hypothesis (\ref{eq:dim M' corank}) and Lemma \ref{le:W coiso} in the appendix that $N$ is coisotropic. It is regular, since $M'$ is regular.
\begin{claim}\label{claim:A CP n M N} We have
\begin{equation}\label{eq:A CP n M N}A(\wt M,N)\geq\pi.
\end{equation}
\end{claim}
\begin{proof}[Proof of Claim \ref{claim:A CP n M N}] It follows from our hypothesis (\ref{eq:int S 2}) that
\begin{equation}\label{eq:int S 2 wt u}\int_{S^2}\wt w^*\wt\om\in\pi\Z,\quad\forall\wt w\in C^\infty(S^2,\wt M).
\end{equation}
Let $\wt u\in C^\infty(\D,\wt M)$ be such that $\wt u(S^1)$ is contained in some isotropic leaf $F$ of $N$. We choose a map $f\in C^\infty(\D,\D)$ that restricts to an orientation preserving diffeomorphism from $B^2_{\frac12}$ to $B^2_1$ and satisfies $f(z)=z/|z|$ on $\D\wo B^2_{\frac12}$. 

The pre-image $\phi^{-1}(F)$ is an isotropic leaf of $M'$. By hypothesis it is simply-connected. Hence the same holds for $F$. It follows that there exists a map $\wt v\in C^\infty(\D,F)$ satisfying $\wt v=\wt u$ on $S^1$. Modifying $\wt v$, we may assume that $\wt v(z)=\wt v(z/|z|)$ for every $z\in\D\wo B^2_{\frac12}$. 

We denote by $\BAR{\D}$ the disk with the reversed orientation and by $\D\#\BAR{\D}$ the smooth oriented manifold obtained by concatenating the two disks along their boundary. We define $\wt w:\D\#\BAR\D\to\CP^n\x M$ to be the concatenation of $\wt u\circ f$ and $\wt v$. This is a smooth map. It follows that
\begin{equation}\label{eq:int D BAR D}\int_{\D\#\BAR\D}\wt w^*\wt\om=\int_\D(\wt u\circ f)^*\wt\om-\int_\D\wt v^*\wt\om=\int_\D\wt u^*\wt\om-0.\end{equation}
Since $\D\cup\BAR{\D}$ is diffeomorphic to $S^2$, (\ref{eq:int S 2 wt u}) implies that $\int_{\D\#\BAR\D}\wt w^*\wt\om\in\pi\Z$. Combining this with (\ref{eq:int D BAR D}), inequality (\ref{eq:A CP n M N}) follows. This proves Claim \ref{claim:A CP n M N}.
\end{proof}
This claim and the hypothesis $\dim M'<2n$ contradict Theorem \ref{thm:action}. Hence the presymplectic embedding $\phi:M'\to\wt M$ does not exist. This proves Corollary \ref{cor:presympl}.
\end{proof}
\appendix
\section{Lifting paths}
The following result was used in the proof of Lemma \ref{le:lift}. Let $M',M$ be smooth manifolds, $f:M'\to M$ a smooth proper submersion, $p'\in M'$, and $x\in C^\infty([0,1],M)$. 
\begin{prop}[lifting a path]\label{prop:lift path} If $f(p')=x(0)$ then there exists a path $x'\in C^\infty([0,1],M')$, such that
\[f\circ x'=x,\quad x'(0)=p'.\]
\end{prop}
The proof of this lemma is based on the following.
\begin{lemma}[locally lifting a path]\label{le:loc lift} Let $t_0\in[0,1]$ and $H\sub TM'$ be a (smooth) subbundle that is complementary to $\ker df$, i.e., satisfies $TM'=H\oplus\ker df$. 
\begin{enui}
\item\label{le:loc lift:ex} (local existence) If $f(p')=x(t_0)$ then there exists a (relatively) open neighbourhood $V$ of $t_0$ in $[0,1]$ and a path $x'\in C^\infty(V,M')$, satisfying
\begin{eqnarray}\label{eq:dot x '}&\dot x'(t)\in H_{x'(t)},\quad f\circ x'(t)=x(t),\quad\forall t\in V,&\\
\label{eq:x' t 0}&x'(t_0)=p'.&
\end{eqnarray}
\item\label{le:loc lift:unique} (local uniqueness) If $V_0,V_1$ are open neighbourhoods of $t_0$ in $[0,1]$ and $x'_0,x'_1\in C^\infty\big([0,1],M'\big)$ are paths, satisfying (\ref{eq:dot x '},\ref{eq:x' t 0}) then there exists an open neighbourhood $V\sub V_0\cap V_1$ of $t_0$ in $[0,1]$, such that $x'_0=x'_1$ on $V$.
\end{enui}
\end{lemma}
\begin{rmk}[global uniqueness]\label{rmk:global} For $i=0,1$ let $t_i\in[0,1]$ and $x'_i\in C^\infty\big([0,t_i],M'\big)$ be a path satisfying (\ref{eq:dot x '}) and $x'_i(0)=p'$. Then $x'_0=x'_1$ on $\big[0,\min\big\{t_0,t_1\big\}\big]$. This follows from Lemma \ref{le:loc lift}(\ref{le:loc lift:unique}).
\end{rmk}
\begin{proof}[Proof of Lemma \ref{le:loc lift}] By using a chart in $M$ we may assume \Wlog that $M=\R^n$. Using the Implicit Function Theorem and our hypothesis that $f$ is a smooth submersion, we may further assume \Wlog that $M'=\R^m\x\R^n$ and $f=\pr_2:\R^m\x\R^n\to\R^n$, the canonical projection.

For $t\in[0,1]$ and $y\in\R^m$ we define $X_t(y)\in\R^m$ to be the unique vector, such that
\begin{equation}\label{eq:X t y}\big(X_t(y),\dot x(t)\big)\in H_{(y,x(t))}.\end{equation}
This vector exists and is unique, since $H$ is complementary to $\ker df=\ker\pr_2$. The family $(X_t)_{t\in[0,1]}$  is a smooth time-dependent vector field on $\R^m$. We write $p'=\big(y_0,x(t_0)\big)$.\\

\textbf{We prove (\ref{le:loc lift:ex}).} By the Picard-Lindel\"of theorem there exist an open neighbourhood $V$ of $t_0$ in $[0,1]$ and a smooth solution $y\in C^\infty\big(V,\R^m\big)$ of the ordinary differential equation
\[\dot y=X_t\circ y,\quad y(t_0)=y_0.\]
Using (\ref{eq:X t y}), the path $x':=(y,x):[0,1]\to M'=\R^m\x\R^n$ satisfies (\ref{eq:dot x '},\ref{eq:x' t 0}). This proves (\ref{le:loc lift:ex}).\\

\textbf{Statement (\ref{le:loc lift:unique})} follows from a similar argument. This proves Lemma \ref{le:loc lift}.
\end{proof}
\begin{proof}[Proof of Proposition \ref{prop:lift path}]\setcounter{claim}{0} We choose a subbundle $H\sub TM'$ that is complementary to $\ker df$. (We may define $H$ to be the normal bundle of $\ker df$ with respect to some Riemannian metric.) We define
\begin{align}\label{eq:y'}y':=\bigcup\big\{x'\,&\big|\,t_1\in[0,1],\,x'\in C^\infty\big([0,t_1],M'\big):\,\\
\nn&(\ref{eq:dot x '})\textrm{ with }V=[0,t_1],\,x'(0)=p'\big\}\sub[0,1]\x M'.\end{align}
It follows from Remark \ref{rmk:global} that there exists $t_0\in[0,1]$ such that $y'$ is a smooth map from $[0,t_0)$ or $[0,t_0]$ to $M'$. Proposition \ref{prop:lift path} is a consequence of the following claim.
\begin{claim}\label{claim:y'} The domain of $y'$ is $[0,1]$. 
\end{claim}
\begin{pf}[Proof of Claim \ref{claim:y'}] We define
\[X:=\big\{(x',v)\,\big|\,v\in T_{f(x')}M\big\},\quad\Phi:X\to H,\,\Phi_{x'}v:=\Phi(x',v):=v',\]
where $v'\in H_{x'}$ is the unique vector satisfying $df(x')v'=v$. Since $df(x')$ is surjective and $T_{x'}M'=H_{x'}\oplus\ker df(x')$, this vector exists and is unique, hence $\Phi$ is well-defined. This map is smooth, since $H$ is smooth. We choose Riemannian metrics $g$ on $M$ and $g'$ on $M'$. Since $f$ is proper, the pre-image $K':=f^{-1}\big(x([0,1])\big)\sub M'$ is compact. Therefore,
\[C:=\sup\big\{\big|\Phi_{x'}\big|\,\big|\,x'\in K'\big\}<\infty.\]
Here $|\Phi_{x'}|$ denotes the operator norm of the linear map $\Phi_{x'}:T_{f(x')}M\to T_{x'}M'$ \wrt the norms induced by $g$ and $g'$. 

Let $t\in[0,t_0)$. By (\ref{eq:dot x '}) we have $\dot y'(t)\in H_{y'(t)}$ and $df(y'(t))\dot y'(t)=\dot x(t)$, and therefore
\[\dot y'(t)=\Phi_{y'(t)}\dot x(t).\]
Since $y'\big([0,t_0)\big)\sub K'$, it follows that
\[\big|\dot y'(t)\big|\leq C\big|\dot x(t)\big|\leq C\max_{t\in[0,1]}\big|\dot x(t)\big|.\]
It follows that $y'(t)$ converges to some point $y'_0$, as $t\uparrow t_0$. 

Assume now by contradiction that the domain of $y'$ is not equal to $[0,1]$. We choose $V,x'$ as in Lemma \ref{le:loc lift}(\ref{le:loc lift:ex}), with $p'$ replaced by $y'_0$. Concatenating $y'$ with $x'$, we obtain a solution $z'$ of (\ref{eq:dot x '}) with $V$ replaced by an interval that strictly contains the domain of $y'$, such that $z'(0)=p'$. (Here we use Lemma \ref{le:loc lift}(\ref{le:loc lift:unique}), which ensures that $x'$ and $y'$ agree on the intersection of $V$ with the domain of $y'$, if we shrink $V$.) By (\ref{eq:y'}) we have $z'\sub y'$.
This is a contradiction. It follows that the domain of $y'$ is equal to $[0,1]$. This proves Claim \ref{claim:y'} and completes the proof of Proposition \ref{prop:lift path}.
\end{pf}
\end{proof}
\section{Coisotropic subspaces of presymplectic vector spaces}
The following lemma was used in the proof of Lemma \ref{le:V V' V''}. Let $(V,\om)$ be a finite-dimensional presymplectic vector space and $W\sub V$ a linear subspace. We denote by $i:W\to V$ the inclusion, and by
\[W^\om:=\big\{v\in V\,\big|\,\om(v,w)=0,\,\forall w\in W\big\}\]
the presymplectic complement of $W$ in $V$.
\begin{lemma}\label{le:W coiso} The subspace $W$ is coisotropic iff 
\begin{equation}\label{eq:dim W}\dim W+\dim W^{i^*\om}\geq\dim V+\dim V^\om.%
\footnote{By definition, $W^{i^*\om}$ is the presymplectic complement of $W$ inside $W$.}%
\end{equation}
\end{lemma}
The proof of this lemma is based on the following.
\begin{lemma}\label{le:dim W om} We have
\begin{equation}\label{eq:dim W om}\dim W+\dim W^\om=\dim V+\dim\big(V^\om\cap W\big).
\end{equation}
\end{lemma}
\begin{rmk}\label{rmk:W i om}Since $W^{i^*\om}\sub W^\om$, Lemma \ref{le:dim W om} implies that inequality ``$\leq$'' in (\ref{eq:dim W}) holds, for every linear subspace $W$.
\end{rmk}
\begin{proof}[Proof of Lemma \ref{le:dim W om}]\setcounter{claim}{0} We define the linear map
\[\flat_\om:V\to V^*,\quad\flat_\om v:=\om(v,\cdot).\]
Then $W^\om=\ker(i^*\flat_\om)$, and therefore,
\begin{equation}\label{eq:dim V}\dim\im(i^*\flat_\om)+\dim W^\om=\dim V.\end{equation}
Consider the canonical isomorphism $\iota:V\to V^{**}$, $\iota(v)(\phi):=\phi(v)$. A direct calculation shows that $(\flat_\om)^*\iota=-\flat_\om$. It follows that $(\flat_\om i)^*\iota=-i^*\flat_\om$, and therefore
\[\dim\im(\flat_\om i)=\dim\im(\flat_\om i)^*=\dim\im\big((\flat_\om i)^*\iota\big)=\dim\im(i^*\flat_\om).\]
Combining this with (\ref{eq:dim V}), we obtain
\begin{align}\nn\dim W+\dim W^\om&=\dim\ker(\flat_\om i)+\dim\im(\flat_\om i)+\dim W^\om\\
\label{eq:dim W dim W om}&=\dim\ker(\flat_\om i)+\dim V.
\end{align}
Since $\ker(\flat_\om i)=V^\om\cap W$, equality (\ref{eq:dim W om}) follows. This proves Lemma \ref{le:dim W om}.
\end{proof}
\begin{proof}[Proof of Lemma \ref{le:W coiso}] \textbf{We prove ``$\then$''.} Assume that $W$ is coisotropic. Then $W^\om\sub W^{i^*\om}$ and therefore, using Lemma \ref{le:dim W om}, we have
\begin{align*}\dim W+\dim W^{i^*\om}&\geq\dim W+\dim W^\om\\
&=\dim V+\dim\big(V^\om\cap W\big).
\end{align*}
Since $V^\om\sub W^\om\sub W$, inequality (\ref{eq:dim W}) follows. This proves ``$\then$''.\\

To prove the opposite implication, assume that (\ref{eq:dim W}) holds. Using Lemma \ref{le:dim W om}, it follows that
\begin{align*}\dim W^\om&=\dim V-\dim W+\dim\big(V^\om\cap W\big)\\
&\leq\dim W^{i^*\om}-\dim V^\om+\dim\big(V^\om\cap W\big)\\
&\leq\dim W^{i^*\om}.
\end{align*}
Since $W^\om\cont W^{i^*\om}$, it follows that $W^\om=W^{i^*\om}\sub W$. Therefore, $W$ is coisotropic. This proves ``$\follows$'' and completes the proof of Lemma \ref{le:W coiso}.
\end{proof}


\begin{thebibliography}{99}

\bibitem[Ab1]{AbNote} C.~Abbas, \emph{A note on V.~I.~Arnold's chord conjecture.} Internat.~Math.~Res.~Notices 1999, {\bf no.~4}, 217--222. 

\bibitem[Ab2]{AbFinite} C.~Abbas, \emph{Finite energy surfaces and the chord problem}, Duke Math.~J.~\textbf{96} (1999), \textbf{no.~2}, 241--316.

\bibitem[Ab3]{AbChord} C.~Abbas, \emph{The chord problem and a new method of filling by pseudoholomorphic curves}, Int.~Math.~Res.~Not.~2004, \textbf{no.~18}, 913--927. 

\bibitem[Ar]{Ar} V.~I.~Arnol'd, \emph{The first steps of symplectic topology}, (Russian) Uspekhi Mat.~Nauk \textbf{41} (1986), \textbf{no.~6 (252)}, 3--18, 229. 

\bibitem[BC1]{BCSympl} P.~Biran, K.~Cieliebak, \emph{Symplectic topology on subcritical manifolds}, Comment.~Math.~Helv.~\textbf{76} (2001), \textbf{no.~4}, 712--753.

\bibitem[BC2]{BCLag} P.~Biran, K.~Cieliebak, \emph{Lagrangian embeddings into subcritical Stein manifolds}, Israel J.~Math.~\textbf{127} (2002), 221--244.

\bibitem[BEE]{BEE} F.~Bourgeois, T.~Ekholm, and Y.~Eliashberg, \emph{Effect of Legendrian surgery}, with an appendix by S.~Ganatra and M.~Maydanskiy, Geom.~Topol.~\textbf{16} (2012), \textbf{no.~1}, 301--389.

\bibitem[Ch1]{ChCrit} Yu.~Chekanov, \emph{Critical points of quasifunctions, and generating families of Legendrian manifolds}, (Russian) Funktsional.~Anal.~i Prilozhen.~\textbf{30} (1996), \textbf{no.~2}, 56--69, 96; translation in Funct.~Anal.~Appl.~\textbf{30} (1996), \textbf{no.~2}, 118--128. 

\bibitem[Ch2]{ChLag} Yu.~Chekanov, \emph{Lagrangian intersections, symplectic energy, and areas of holomorphic curves}, Duke Math. J. {\bf 95} (1998), {\bf no. 1}, 213-226. 

\bibitem[Ch3]{ChDiff} Yu.~Chekanov, \emph{Differential algebra of Legendrian links}, Invent.~Math.~\textbf{150} (2002), \textbf{no.~3}, 441--483. 

\bibitem[Ci]{Ci} K.~Cieliebak, \emph{Handle attaching in symplectic homology and the chord conjecture}, J.~Eur.~Math.~Soc.~\textbf{4} (2002), \textbf{no.~2}, 115--142. 

\bibitem[CM]{CM} K.~Cieliebak and K.~Mohnke, \emph{Punctured holomorphic curves and Lagrangian embeddings}, arXiv:1411.1870v2.

\bibitem[Ek]{Ek} T.~Ekholm, \emph{Rational SFT, linearized Legendrian contact homology, and Lagrangian Floer cohomology}, Perspectives in analysis, geometry, and topology, 109--145, Progr.~Math., \textbf{296}, Birkh\"auser/Springer, New York, 2012. 

\bibitem[EESa]{EESa} T.~Ekholm, J.~B.~Etnyre, J.~Sabloff, \emph{A duality exact sequence for Legendrian contact homology}, Duke Math.~J.~\textbf{150} (2009), \textbf{no.~1}, 1--75. 

\bibitem[EESu]{EESu} T.~Ekholm, J.~B.~Etnyre, M.~Sullivan, \emph{Legendrian contact homology in $P\x\R$}, Trans.~Amer.~Math.~Soc.~\textbf{359} (2007), \textbf{no.~7}, 3301--3335. 

\bibitem[EGH]{EGH} Y.~Eliashberg, A.~Givental, H.~Hofer, \emph{Introduction to symplectic field theory}, Geom.~Funct.~Anal.~2000, Special Volume, Part \textbf{II}, 560--673. 

\bibitem[GG]{GG} V.~L.~Ginzburg and B.~G\"urel, \emph{A $C^2$-smooth counterexample to the Hamiltonian Seifert Conjecture in $\R^4$}, Ann.~Math., Second Series, Vol.~\textbf{158}, no.~\textbf{3} (2003), 953--976.

\bibitem[Giv]{Giv} A.~B.~Givental', \emph{The nonlinear Maslov index}, Geometry of low-dimensional manifolds, 2 (Durham, 1989), 35--43, London Math.~Soc.~Lecture Note Ser.~{\bf 151}, Cambridge Univ.~Press, Cambridge, 1990. 

\bibitem[HT1]{HT1} M.~Hutchings and C.H.~Taubes, \emph{Proof of the Arnold chord conjecture in three dimensions, I}, Math.~Res.~Lett.~\textbf{18} (2011), \textbf{no.~2}, 295--313.

\bibitem[HT2]{HT2} M.~Hutchings and C.H.~Taubes, \emph{Proof of the Arnold chord conjecture in three dimensions, II}, Geom.~Topol.~\textbf{17} (2013), \textbf{no.~5}, 2601--2688.

\bibitem[Le]{Le} J.~M.~Lee, \emph{Introduction to smooth manifolds}, second ed., Graduate Texts in Mathematics, \textbf{218}. Springer, New York, 2013. xvi+708 pp.

\bibitem[Me]{Me} W.~J.~Merry, \emph{Lagrangian Rabinowitz Floer homology and twisted cotangent bundles}, Geom.~Dedicata \textbf{171} (2014), 345--386. 

\bibitem[Mo]{Mo} K.~Mohnke, \emph{Holomorphic disks and the chord conjecture}, Ann.~of Math.~{\bf(2) 154} (2001), {\bf no.~1}, 219--222. 

\bibitem[Rit]{Rit} A.~Ritter, \emph{Topological quantum field theory structure on symplectic cohomology}, J.~Topol.~\textbf{6} (2013), \textbf{no.~2}, 391--489. 

\bibitem[Sa1]{SaIt} S.~Sandon, \emph{On iterated translated points for contactomorphisms of $R^{2n+1}$ and $R^{2n}\x S^1$}, Internat.~J.~Math.~\textbf{23} (2012), \textbf{no.~2}, 1250042, 14 pp. 

\bibitem[Sa2]{SaMorse} S.~Sandon, \emph{A Morse estimate for translated points of contactomorphisms of spheres and projective spaces}, Geom.~Dedicata \textbf{165} (2013), 95--110. 

\bibitem[Se]{Se} P.~Seidel, \emph{Graded Lagrangian submanifolds}, Bull.~Soc.~Math.~France \textbf{128} (2000), \textbf{no.~1}, 103--149.

\bibitem[SZ]{SZ} J.~Swoboda and F.~Ziltener, \emph{Coisotropic Displacement and Small Subsets of a Symplectic Manifold}, Math.~Z., Vol.~{\bf 271}, Iss.~{\bf 1} (2012), p.~415--445.

\bibitem[Zi]{Zi} F.~Ziltener, \emph{Coisotropic Submanifolds, Leafwise Fixed Points, and Presymplectic Embeddings}, J.~Symplectic Geom.~{\bf 8} (2010), no.~{\bf 1}, 1--24.

\end{thebibliography}
\end{document}